\documentclass[runningheads]{llncs}
\usepackage{graphicx}
\usepackage{epstopdf}
\usepackage{multirow}

\begin{document}
\title{Optimal Investment in the Development of Oil and Gas Field\thanks{The research
is carried out within the framework of the state contract of the
Sobolev Institute of Mathematics (project 0314--2019--0014).}}
\author{Adil Erzin\inst{1,2}\orcidID{0000-0002-2183-523X} \and
Roman Plotnikov\inst{1}\orcidID{0000-0003-2038-5609} \and Alexei
Korobkin\inst{3} \and Gregory Melidi\inst{2} \and Stepan
Nazarenko\inst{2}}
\authorrunning{A. Erzin et al.}
\institute{Sobolev Institute of Mathematics, SB RAS, Novosibirsk 630090, Russia \and
Novosibirsk State University, Novosibirsk 630090, Russia \and Gazpromneft, St. Petersburg, Russia\\
\email adilerzin@math.nsc.ru}
\maketitle              
\begin{abstract}
Let an oil and gas field consists of clusters in each of which an
investor can launch at most one project. During the implementation
of a particular project, all characteristics are known, including
annual production volumes, necessary investment volumes, and profit.
The total amount of investments that the investor spends on
developing the field during the entire planning period we know. It
is required to determine which projects to implement in each cluster
so that, within the total amount of investments, the profit for the
entire planning period is maximum.

The problem under consideration is NP-hard. However, it is solved by
dynamic programming with pseudopolynomial time complexity.
Nevertheless, in practice, there are additional constraints that do
not allow solving the problem with acceptable accuracy at a
reasonable time. Such restrictions, in particular, are annual
production volumes. In this paper, we considered only the upper
constraints that are dictated by the pipeline capacity. For the
investment optimization problem with such additional restrictions,
we obtain qualitative results, propose an approximate algorithm, and
investigate its properties. Based on the results of a numerical
experiment, we conclude that the developed algorithm builds a
solution close (in terms of the objective function) to the optimal
one. 
\end{abstract}
\section{Introduction}
The founder of the mathematical theory of portfolio optimization is
G. Markowitz, who, in 1952, published an article \cite{Markowitz52}
with the basic definitions and approaches for evaluating investment
activity. He developed a methodology for the formation of an
investment portfolio, aimed at the optimal choice of assets, based
on a given ratio of profitability/risk. The ideas formulated by him
form the basis of modern portfolio theory
\cite{Markowitz52,Markowitz59}.

The author of \cite{Pliska86} gave a review of portfolio selection
methods and described the prospects of some open areas. At first,
the author described the classical Markowitz model. Then comes the
``intertemporal portfolio choice'' developed by Merton
\cite{Merton69,Merton71}, the fundamental concept of dynamic hedging
and martingale methods. Pliska \cite{Pliska86}, Karatzas
\cite{Karatzas87}, as well as Cox and Huang \cite{Cox86} made the
main contribution to the development of this direction. The authors
of \cite{Detemple03} and \cite{Ocone91} proposed the formulas for
the optimal portfolio for some private productions. These formulas
have the form of conditional expectation from random variables.

In most well-known studies, the problem of optimal investment is
solved numerically \cite{Brandt05,Brennan97}, which does not allow
us to identify the contribution of portfolio components to the
optimal solution. In \cite{Brandt04}, a new approach is proposed for
dynamic portfolio selection, which is not more complicated than the
Markowitz model. The idea is to expand the asset space by including
simple, manageable portfolios and calculate the optimal static
portfolio in this extended space. It is intuitively assumed that a
static choice among managed portfolios is equivalent to a dynamic
strategy.

If we consider investing in specific production projects, then each
of them is either implemented or not. In contrast to the classical
Markowitz's problem, a discrete statement, arises, and the
mathematical apparatus developed for the continuous case is not
applicable. In \cite{Malah16}, the authors examined a two-criterion
problem of maximizing profit and minimizing risk. The
characteristics of each project, mutual influence, and the capital
available to the investor are known. For the Boolean formulation of
the problem, the authors proved NP-hardness and found special cases
when the problem is solved with pseudopolynomial time complexity.

The portfolio optimization problems described above relate to the
stock market. For companies operating in the oil and gas sector,
optimization problems are relevant. In these problems it is
necessary to maximize total profit and minimize risks for a given
period, taking into account additional restrictions, for example, on
production volume, as well as problems in which it is necessary to
maximize production (or profit) for a given amount of funding. In
\cite{Akopov04}, the author presented an approach aimed at improving
the efficiency of the management of the oil and gas production
association. Two control loops are distinguished: macroeconomic,
which is responsible for optimizing policies at the aggregated level
(industry and regional), and microeconomic, which is responsible for
optimizing the organizational and functional structure of the
company. The first circuit implemented using the author developed
computable models of general economic equilibrium and integrated
matrices of financial flows. The second circuit performed using an
approach based on simulation of the business processes of an
enterprise.

The author of \cite{Dominikov17} considers the problem of forming a
portfolio of investment projects, which required to obtain maximum
income under given assumptions regarding risks. A method is proposed
based on a comprehensive multidimensional analysis of an investment
project. The authors of \cite{Goncharenko08,Konovalov06} consider a
problem of minimizing the deposit costs with restrictions on the
volume of the production. They propose an algorithm for building an
approximate solution by dynamic programming. The accuracy of the
algorithm depends on the discretization step of the investment
volume. The authors of \cite{Goncharenko08} formulate the problem of
minimizing various costs associated with servicing wells, with
limitations related to the amount of oil produced, as a linear
programming problem, and find the optimal solution using the simplex
method.

For the decision-maker, the main concern is how to allocate limited
resources to the most profitable projects. Recently, a new
management philosophy, Beyond NPV (Net Present Value), has attracted more and more
international attention. Improved portfolio optimization model
presented in \cite{Qing14}. It is an original method, in addition to
NPV, for budgeting investments. In the proposed model, oil company
executives can compromise between profitability and risk concerning
their acceptable level of risk. They can also use the ``operating
bonus'' to distinguish their ability to improve the performance of
major projects. To compare optimized utility with non-optimized
utility, the article conducted a simulation study based on 19
foreign upstream assets owned by a large oil company in China. The
simulation results showed that the optimization model, including the
``operating bonus'', is more in line with the rational demand of
investors.

The purpose of the paper \cite{Bulai18} is to offer a tool that
might support the strategic decision-making process for companies
operating in the oil industry. Their model uses Markowitz's
portfolio selection theory to construct an efficient frontier for
currently producing fields and a set of investment projects. These
relate to oil and gas exploration projects and projects aimed at
enhancing current production. The net present value obtained for
each project under a set of user-supplied scenarios. For the
base-case scenario, the authors also model oil prices through Monte
Carlo simulation. They run the model for a combination of portfolio
items, which include both currently producing assets and new
exploration projects, using data characteristics of a mature region
with a high number of low-production fields. The objective is to
find the vector of weights (equity stake in each project), which
minimizes portfolio risk, given a set of expected portfolio returns.

Due to the suddenness, uncertainty, and colossal loss of political
risks in overseas projects, the paper \cite{Huang19} considers the
time dimension and the success rate of project exploitation for the
goal of optimizing the allocation of multiple objectives, such as
output, investment, efficiency, and risk. A linear portfolio risk
decision model proposed for multiple indicators, such as the
uncertainty of project survey results, the inconsistency of project
investment time, and the number of projects in unstable political
regions. Numerical examples and the results test the model and show
that the model can effectively maximize the portfolio income within
the risk tolerance range under the premise of ensuring the rational
allocation of resources.

This paper discusses the problem of optimal investment of oil and
gas field development consisting of subfields -- \emph{clusters}.
For each cluster, there are several possibilities for its development,
which we call the \emph{projects}. Each project characterized by
cost, lead time, resource intensity, annual production volumes, and
profit from its implementation. Also, there are restrictions on the
annual production volumes of the entire field. This requirement
leads to the need for a later launch of some projects so that the
annual production volume does not exceed the allowable volumes.
Assuming that a project launched later is another project, we
proposed a statement of the problem in the form of a Boolean linear
programming (BLP) problem. We estimated the maximum dimension at
which CPLEX solves the BLP in a reasonable time. For a
large-dimensional problem, we developed a method that constructs an
approximate solution in two stages. At the first stage, the problem
is solved without limitation on the volume of annual production.
This problem remains NP-hard, but it is solvable by a
pseudopolynomial dynamic programming algorithm. As a result, one
project is selected for each cluster. The project is characterized,
in particular, by the year of launch and production volumes in each
subsequent year. If we start the project later, the annual
production volumes shift. In the second stage, the problem of
determining the start moments of the projects selected at the first
stage is solved, taking into account the restrictions on annual
production volumes, and the profit is maximal. We developed a local
search algorithm for partial enumeration of permutations of the
order in which projects are launched. At the same time, for each
permutation, the algorithm of tight packing of production profiles
developed by us (we call it a greedy algorithm), which builds a
feasible solution, is applied. A numerical experiment compared our
method and CPLEX.

The rest of the paper has the following organization. In Section 2,
we state the problem as a BLP. In Section 3, the problem without
restrictions on the volume of production reduced to a nonlinear
distribution problem, which is solved by dynamic programming. As a
result, a ``best'' project found for each cluster. Section 4
describes the method for constructing an approximate solution by
searching for the start times for the ``best'' projects. The next
section presents the results of a numerical experiment. We identify
the maximum dimension of the problem, which is solved by the CPLEX
package in a reasonable time, and compare the accuracy of the
developed approximate algorithms. Section 6 contains the main
conclusions and describes the directions for further research.

\section{Formulation of the problem}
For the mathematical formulation of the problem, we introduce the
following notation for the parameters:
\begin{itemize}
  \item $[1,T]$ is the planning period;
  \item $C$ is the total amount of investment;
  \item $K$ is the set of clusters ($|K|=n$);
  \item $P_k$ is the set of projects for the development of the cluster $k\in K$ ($\max\limits_k|P_k|=p$)
  taking into account the shift at the beginning of each project;
  \item $d_k^i(t)$ is the volume of production in the cluster $k\in K$ per year $t=1,\ldots,T$,
  if the project $i\in P_k$ is implemented there;
  \item $q_k^i$ is the profit for the entire planning period from the implementation of project $i$ in cluster $k$;
  \item $c_k^i$ is the cost of implementing project $i$ in cluster $k$;
  \item $D(t)$ is the maximum allowable production per year $t$;
\end{itemize}
and for the variables:
$$
x_k^i=\left\{
        \begin{array}{ll}
          1, & \hbox{if project $i$ is selected for cluster $k$;} \\
          0, & \hbox{else.}
        \end{array}
      \right.
$$
Then the problem under consideration can be written as follows.
\begin{equation}\label{e1}
  \sum\limits_{k\in K}\sum\limits_{i\in P_k}q_k^ix_k^i\rightarrow\max\limits_{x_k^i\in\{0,1\}};
\end{equation}
\begin{equation}\label{e2}
  \sum\limits_{k\in K}\sum\limits_{i\in P_k}c_k^ix_k^i\leq C;
\end{equation}
\begin{equation}\label{e3}
  \sum\limits_{i\in P_k}x_k^i\leq 1,\ k\in K;
\end{equation}
\begin{equation}\label{e4}
  \sum\limits_{k\in K}\sum\limits_{i\in P_k}d_k^i(t)x_k^i\leq D(t),\ t\in [1,T].
\end{equation}

\begin{remark}
Each project has various parameters, among which the annual
production volumes. If we start the project later, then the graphic
of annual production will shift entirely. Suppose $d_k^i(t)$ is the
volume production per year $t$ if the project $i$ is implemented in
the cluster $k$. If this project is launched $\tau$ years later, the
annual production during year $t$ will be $d_k^i(t-\tau)$. So, each
project in the set $P_k$ is characterized, in particular, by its
beginning.

However, not all characteristics retain their values at a later
launch of the project. Profit from project implementation depends on
the year of its launch, as money depreciates over the years. One way
to account for depreciation is to use a discount factor. The value
of money decreases with each year by multiplying by a discount
factor that is less than 1. In this regard, at the stage of
preliminary calculations, we recount values associated with
investment and profit.

As a result, the set $P_k$ consists of the initial projects, and the
shifted projects for different years as well. So, having solved the
problem (\ref{e1})--(\ref{e4}), we will choose for each cluster not
only the best project but also the time of its start.
\end{remark}

Problem (\ref{e1})--(\ref{e4}) is an NP-hard BLP. For the dimension
which we define in Section 5, a software package, for example,
CPLEX, can be used to solve it. In order to solve the problem of a
large dimension, it is advisable to develop an approximate
algorithm. To do this, in the next section we consider the problem
(\ref{e1})--(\ref{e3}).

\section{The problem without restrictions on production volumes} \label{sDP}
If there are no restrictions on production volumes, then instead of
the variables $x_k^i$, we can use the variables $c_k$, which are
equal to the amount of the investment allocated for the development
of the cluster $k$. To do this, for each cluster $k$, we introduce a
new profit function $q_k(c_k)$, which does not depends on the
selected project but depends on the amount of investment. For each
$k$, the function $q_k(c_k)$ is obviously non-decreasing piecewise
constant. Moreover, if we know the value $c_k$, then one project is
uniquely will be used to develop the cluster $k$, and all its
characteristics will be known. Indeed, the more money is required to
implement the project, the more efficient it is (more profitable).
If this is not a case, then a less effective but more expensive
project can be excluded. Obviously, the values of all functions
$q_k(c_k)$, $c_k\in [0,C]$, $k\in K$, are not difficult to calculate
in advance. The complexity of this procedure does not exceed
$O(KpC)$.

Given the previous, we state the problem of maximizing profit
without restrictions on production volumes, assuming that all
projects start without delay,  in the following form.
\begin{equation}\label{e5}
  \sum\limits_{k\in K}q_k(c_k)\rightarrow\max\limits_{c_k\in [0,C]};
\end{equation}
\begin{equation}\label{e6}
  \sum\limits_{k\in K}c_k\leq C.
\end{equation}

Although problem (\ref{e5})--(\ref{e6}) become easier than the
problem (\ref{e1})--(\ref{e4}), it remains NP-hard. However, it is a
distribution problem, for the solution of which we apply the dynamic
programming method, the complexity of which is $O(n(C/\delta)^2)$,
where $\delta$ is the step of changing the variable $c_k$. Solving
the problem (\ref{e5})--(\ref{e6}), we choose the ``best'' project
for each cluster. If it turns out that at the same time, all the
restrictions (\ref{e4}) fulfilled, then this solution is
\emph{optimal} for the original problem (\ref{e1})--(\ref{e4}). If
at least one inequality (\ref{e4}) violated, then we will construct
a feasible solution in the manner described in the next section.

\section{Consideration of restrictions on production volumes}
We will not change the projects selected for each cluster as a
result of  solving the problem (\ref{e5})--(\ref{e6}). We will try
to determine the moments of launching these projects so that
inequalities (\ref{e4}) fulfill, and profit takes maximal value. The
project selected for the cluster $k$ is characterized by the
production volumes $d_k(t)$ in each year $t\in [1,T]$. It is
necessary to shift the beginning of some projects to a later time so
that in each year $t\in [1,T]$ the total production is at most
$D(t)$:
$$
\sum\limits_{k\in K} d_k(t)\leq D(t).
$$

Assume that the cluster $k$ development project, whose beginning is
shifted by $i\in [0,t_k]$ years, is another project. Then for each
cluster, there is a set of projects, which we denote as before by
$P_k$ ($|P_k|=t_k+1$). As a result, to determine the shift in the
start of the project launch for each cluster, it is enough to solve
the problem (\ref{e1})--(\ref{e4}) without restriction (\ref{e2}),
in which the Boolean variable $x_k^i = 1$ if and only if the start
of the cluster project $k$ is shifted by $i$ years. Then the
solution to the small-dimension problem can be found using a CPLEX.
However, for a large-sized problem, it is necessary to use an
approximate algorithm.

\subsection{Greedy algorithm}
Suppose we order the projects according to the years of their
launch. A shift in the start of projects changes this order. The
order in which projects start uniquely determined by the permutation
$\pi$ of the cluster numbers $\{1,2,\ldots,n\}$. For a given
permutation, we describe informally a greedy algorithm that
constructs a feasible solution to the problem.

Denote by $P(\pi)$ the list of ordered projects. The first project
starts  without delay (with zero shift). We exclude it from the set
$P(\pi)$. For the first project of the updated set $P(\pi)$, we
determine its \emph{earliest} start time, which is no less than the
start time of a previous project, to comply with the production
order and restrictions (\ref{e4}) in each year and exclude this
project from the set $P(\pi)$. We continue the process until the
start year of the last project, $\pi(n)$ is found.

The greedy algorithm will construct a feasible solution for the
given  permutation $\pi$, if it exists, with the time complexity of
$O(nT)$. In the oil and gas industry, profiles (graphs) of annual
production volumes have a log-normal distribution \cite{Power92},
which is characterized by a rapid increase, and then a slight
decrease. This observation and the following lemma, to some extent,
justify why we use the greedy algorithm.

\begin{proposition}
If the order of launching the projects is known, the annual
production schedules for all projects are not-increasing, and
$D(t)=D=const$, $t\in [1, T]$, then the greedy algorithm determines
the optimal start years for all projects.
\end{proposition}

\begin{proof}
In the problem under consideration, time is discrete (measured in
years).  Therefore, the value of production in each cluster is a
certain real number that does not change for one year. A greedy
algorithm for a given order of projects determines the earliest
start time for each project, which is not less than the start time
of the previous project. Suppose that in all optimal solutions,
there is at least one project that begins later than the year
determined by the greedy algorithm. Consider some optimal solution
and let $k$ be the first project that we can start earlier (Fig.
~\ref{fig1}\emph{a}). Since the project $k$ can start earlier, then
move it as much as possible to the left to maintain validity (Fig.
~\ref{fig1}\emph{b}). Notice that it is enough to check the value of
production $d_k^1$ only in the first year of the project $k$ because
it is not less than production in subsequent years ($d_k^t\leq
d_k^1,\ t>1$). The solution obtained after shifting the project $k$
to the left is no worse (and taking into account the discount
coefficient, even better), but the project $k$ starts earlier, which
contradicts the assumption. The proof is over.
\end{proof}
\begin{figure}
\centering
\includegraphics[bb= 0 0 500 250, clip, scale=0.4]{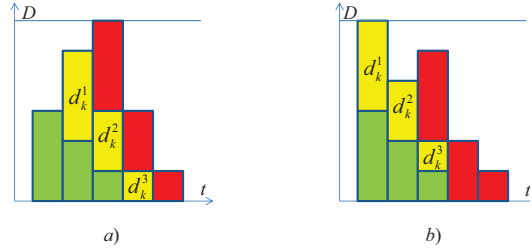}
\caption{Illustration to the Lemma 1 proof (Bars of the same
color belong to the same project, and the height of the bar is
the volume of production in the corresponding year). $a$) Project
$k$ (yellow) may start earlier; $b$) After shifting the project $k$
to the left.} \label{fig1}
\end{figure}

So, with a particular order of projects, a greedy algorithm builds a
solution close to optimal with $O(nT)$ time complexity. A complete
enumeration of permutations requires $O(n!)$ operations. However, it
is reasonable to develop a local search algorithm in which, at each
step, the best permutation is searched in the vicinity of the
current permutation. In order to obtain a solution for a given
permutation, a greedy algorithm is used. The higher the profit in a
particular order of projects, the better the permutation. In the
next subsection, we develop a local search algorithm.

\subsection{Local search}
Using the greedy algorithm described in the previous subsection, one
can construct a  solution for each permutation of the cluster
numbers. Therefore, it is essential to find a permutation where the
solution constructed in such a way is near-optimal. For this reason,
we suggest a local search procedure for permutations starting from
some promising one.

In order to obtain the first permutation for the local search procedure we perform the following greedy algorithm. At the first step we choose such cluster that yields maximum value of income if its development is started at the first year. The number of this cluster becomes the first value of the permutation. Then, at each step of the algorithm, we choose among the unprocessed clusters such cluster, that if its development is started at the earliest year (taking into account the per-year production bound and already chosen clusters), then the total income increment will be maximum. After such cluster is found, we assign the corresponding shift for it and set its number to the next permutation value. The permutation obtained by the described greedy procedure becomes the first permutation of the local search algorithm.

For each permutation in the local search procedure, we construct a
solution using the greedy algorithm, that is described in previous subsection, with time complexity $O(nT)$. As
a movement operation of the local search procedure, we perform the
best possible exchange of two different elements of a permutation.
Then the cardinality of the neighborhood of the current permutation
is $O(n^2)$, and the time complexity of the searching the best
solution is this neighborhood is $O(n^3T)$.

\section{Simulation}
The proposed algorithms have been implemented in the C++ programming
language and launched on the randomly generated test instances. We
also used the IBM ILOG CPLEX package (version 12.10) in order to
obtain optimal or near-optimal solutions together with guaranteed
upper bounds for the comparison. The numerical experiment performed
on an Intel Core i5-4460 (3.2GHz) 8Gb machine.

For the generation of the test data, we supposed that the
distribution function of production volumes by the planned period is
log-normal \cite{Power92}. The parameters of the distribution
density $\mu$ and $\sigma$ were chosen randomly with uniform
distribution on the intervals $[1,2]$ and $[1,1.4]$ correspondingly.
We defined the maximum value of the production volume for each
project at random with uniform distribution on the interval
$[30,200]$ (in thousands of tons) and then multiplied each per-year
volume by the corresponding scaling factor.

We assumed that the profit is proportional to the production volume.
For each project, in order to generate the profit per each year, we,
at first, took the random coefficient (the cost of one ton in
millions of rubles) uniformly distributed on the interval $[4,6]$.
This value may vary depending on the differences in the condition of
the production, overhead costs, remoteness of the cluster. After
that, we multiplied the per-year production volumes by this
coefficient. We also added the noise to the generated values
multiplying them by the random values uniformly distributed on the
interval $[0.95,1.05]$.

Like the investments, we generated random values uniformly
distributed on the interval $[250, 1500]$ (in millions of rubles).
We assumed that the obtained amount of money spent in the first year
of the project exploitation. In about 10 percent of the cases, other
investments made in the second year of the project exploitation. The
second investment is taken as a random part of the first investment
from 10 to 50 percent. As the upper bounds of investments and
per-year production volumes, we took the one-third part of the sum
by clusters of the maximum values per project and per year. That is,
$$
C = 1/3 \sum\limits_{k\in K}\max\limits_{i \in P_k}c_k^i,
$$
and
$$
D = 1/3 \sum\limits_{k\in K}\max\limits_{i \in P_k, t \in [1, T]}d_k^i(t), D(t) = D\; \forall t \in [1, T].
$$
For solving the problem (\ref{e5})-(\ref{e6}) we set $\delta$ equal
ten thousands of rubles, because, according to our preliminary
experiments, further decrease of $\delta$ does not improve the
solution significantly.

We generated instances for four different variants of the number of
clusters: $n=$ 10, 25, 50, and 100. For each value of $n$, we
generated four instances with different maximum and a minimum number
of projects per cluster: 1) from 1 to 10; 2) from 10 to 25; 3) from
25 to 50, and 4) from 50 to 100. We launched our algorithm and CPLEX
on each instance. The results presented in Table \ref{t1}. In this
table, CPLEX stands for the results obtained by CPLEX launched on
the problem (\ref{e1})-(\ref{e4}). $CPLEX_{fp}$ stands for the
results obtained by CPLEX launched on the restricted problem with
the fixed project per each cluster found by the dynamic programming
method described in Section \ref{sDP}. Notations \emph{obj},
\emph{ub}, and \emph{gap} stand for, correspondingly, the objective
function of the incumbent, the upper bound of the value of objective
function, and the relative difference between \emph{obj} and
\emph{ub}. \emph{decline} stands for the decline (in percents) of
the objective function value of the incumbent of the problem with
the fixed set of projects concerning the objective function of the
incumbent of the entire problem. The last four columns represent the
results obtained by our algorithm, which is named $A$ in the table.
$r_1$ denotes the ratio $obj(A)/ub(CPLEX)$, and $r_2$ denotes the
ratio $obj(A)/ub(CPLEX_{fp})$. The last column stands for the total
running time of our algorithm. The running time of CPLEX was limited
by 60 seconds for all the cases except the last one of the largest size,
 --- in the last case CPLEX was given for 1 hour. It also should be noted that CPLEX was parallelized
on four threads.

As it follows from the table, in the cases of small and moderate size CPLEX solves the problem rather
precisely  within 60 seconds. In these cases it always constructs a solution on
which the value of the objective function differs from the optimal
one by at most 1 percent. Algorithm $A$ constructs a less accurate
solution. As it is seen at the column $r_1$, in the worst case, the
objective value of the obtained solution differs from the optimal by
13 percent, in the best case --- by 3 percent, and on average, this
difference does not exceed 9 percent. As one can see at the column
$decline$, the choice of the projects obtained by solution of the
problem without restriction on the production volumes deteriorates
the solution of the entire problem by up to 12.5 percent. On
average, this decline is about 6 percent. The quality of our local
search procedure applied to the solution obtained by the greedy
heuristic is estimated in the column $r_2$. On average, the ratio
does not exceed 3 percent. In a case of large size, when the number of clusters  
is 250 and the number of projects in each cluster varies from 250 to 500, CPLEX failed 
to construct any feasible solution within 1 hour, but the algorithm $A$ constructed an
approximate solution within about 6 minutes. When we set the projects found by algorithm
$A$ to CPLEX for this instance, it successfully found the solution with rather small value
of gap (less than 0.1 percent) within 1 hour. In this case, the local search procedure found a 
solution that differs from the optimal one by not more than 4 percent.

\begin{table}[!hbtp]
\resizebox{0.8\textwidth}{!}
 {
 \begin{minipage}{\textwidth}
\begin{tabular}{|c|c|c|ccc|cccc|cccc|}
\hline

\multirow{2}{*}{$n$} &\multirow{2}{*}{$p_{\min}$} &\multirow{2}{*}{$p_{\max}$} &\multicolumn{3}{c|}{$CPLEX$} & \multicolumn{4}{c|}{$CPLEX_{fp}$} & \multicolumn{4}{c|}{$A$}\\

& & & \emph{obj} &\emph{ub} & \emph{gap} & \emph{obj} &\emph{ub}    & \emph{gap} & \emph{decline} (\%) &  obj   & $r_1$ & $r_2$     & time (sec.)\\

\hline
\multirow{4}{*}{10}
        &1  &10 &12851.93   &12851.93   &0  &12072  &12072  &0  &6.07   &12072  &0.94   &1  &0.006\\
        &10 &25 &16460.95   &16460.95   &0  &15502.23   &15502.23   &0  &5.82   &14710.6    &0.89   &0.95   &0.005\\
        &25 &50 &16988.53   &16988.53   &0  &14862.28   &14862.28   &0  &12.51  &14764.3    &0.87   &0.99   &0.004\\
        &50 &100    &17140.36   &17140.36   &0  &15607.92   &15607.92   &0  &8.94   &15374.1    &0.9    &0.99   &0.003\\

\hline

\multirow{4}{*}{25}
    &1  &10 &30465.26   &30465.26   &0  &29849.9    &29849.9    &0  &2.02   &29571.9    &0.97   &0.99   &0.082\\
    &10 &25 &42501.57   &42501.57   &0  &39728.94   &39728.94   &0  &6.52   &38930.6    &0.92   &0.98   &0.077\\
    &25 &50 &46508.15   &46849.94   &0.007  &43646.48   &43646.48   &0  &6.15   &42407.4    &0.9    &0.97   &0.056\\
    &50 &100    &47432.09   &47906.72   &0.01   &44307.47   &44307.47   &0  &6.59   &43324  &0.9    &0.98   &0.023\\

\hline
\multirow{4}{*}{50}
    &1  &10 &70568.99   &70568.99   &0  &69609.18   &69609.18   &0  &1.36   &66328.2    &0.94   &0.95   &0.752\\
    &10 &25 &86529.19   &86659.92   &0.002  &80616.31   &80778.51   &0.002  &6.83   &76845.5    &0.89   &0.95   &0.627\\
    &25 &50 &93928.28   &94290.32   &0.003  &88415.05   &88415.05   &0  &5.87   &86861.8    &0.92   &0.98   &0.612\\
    &50 &100    &95201.48   &95621.93   &0.004  &88532.86   &88661.56   &0.001  &7  &86380.7    &0.9    &0.97   &0.39\\

\hline
\multirow{4}{*}{100}
    &1  &10 &139928.34  &140023.11  &0.0007 &136420.24  &136586.12  &0.001  &2.51   &128679 &0.92   &0.94   &5.22\\
    &10 &25 &173898.61  &174065.77  &0.001  &163833.54  &163932.06  &0.0006 &5.79   &154452 &0.89   &0.94   &5.35\\
    &25 &50 &189722.3   &190000.79  &0.001  &177064.12  &177138.3   &0.0004 &6.67   &171431 &0.9    &0.97   &5.84\\
    &50 &100    &195223.21  &195686.39  &0.0023 &180375.66  &180476.58  &0.0005 &7.61   &176830 &0.9    &0.98   &9.44\\

\hline

    250 & 250 &500    &---  &---  &--- &515525.8  &516003.7  &0.0009 &---   &495366 &---    &0.96   &366.4\\
\hline

\end{tabular}
\medskip
\end{minipage}}
\caption{Comparison of the proposed algorithm $A$ with CPLEX}
\label{t1}
\end{table}

\section{Conclusion}
In this paper, we studied the NP-hard problem of maximizing profit
by  choosing long-term cluster development projects within the oil
and gas field, with restrictions on the total investment and maximum
annual production. We proposed a statement of the problem in the
form of Boolean linear programming and set ourselves three goals.
First, to investigate the effectiveness of application software
packages, such as CPLEX, for solving the BLP problem. Secondly,
develop a fast approximate algorithm. Thirdly, compare the
effectiveness of the CPLEX package and the approximate algorithm.

The approximate algorithm consists of two stages, which are
partially dictated by the specifics of the problem. At the first
stage, profit is maximized by selecting one project for each cluster
without taking into account the restrictions on the volume of annual
production. The distribution problem arising, in this case, is
solved by the dynamic programming algorithm with acceptable running
time. Projects selected at the first stage can be launched later
(with a delay). Therefore, at the second stage, the moments of the
start of the selected projects are determined in such a way that the
annual production volumes do not exceed the set values, and the
profit is maximum. The problem of the second stage also formulated
in the form of the BLP. It makes sense without the first stage
because, in practice, development projects for each cluster often
known, and it is only necessary to determine the moments of their
launch.

Production profiles have a characteristic shape, which is determined
by the log-normal distribution law and has the form of a graph that
first overgrows, reaches its maximum value, and then slowly
decreases \cite{Power92}. With a certain degree of assumption, we
assumed that the profiles are non-increasing. We proved that in the
case of non-increasing production profiles and for a given order of
project start (which is determined by the permutation of cluster
numbers), the greedy algorithm constructs the optimal solution. The
algorithm of permutations sorting is justified, and the greedy
algorithm used for each permutation. Iterating over all permutations
is time-consuming, and for large dimensions, it is just not
applicable, so we used a relatively simple local search algorithm.

The results of the numerical experiment on randomly generated
examples surprised us (see Table \ref{t1}). For $25\leq n\leq 100$,
the CPLEX package was not able to build an optimal solution, but it
turned out that CPLEX within one minute builds a feasible solution
quite close to the optimal one. The approximate algorithm that we
developed also builds a solution close to optimal, but CPLEX turned
out to be more efficient for such dimension. Thus, we conclude that for the considered
problem when $n\leq 100$, it is advisable to use a package of application programs
CPLEX instead of our algorithm. In a case of large size, for example when $n\geq 250$, CPLEX failed
to construct any feasible solution within 1 hour, but the algorithm $A$ constructed an
approximate solution within 6 minutes. When we set the projects found by solving the problem 
(\ref{e5})--(\ref{e6}) with $n=250$ to CPLEX, it successfully found the solution with gap less than 0.1\%
within 1 hour. 

Perhaps the situation will change if we consider some additional
restrictions.  In practice, it is necessary to produce annually at
least a given volume and no more than a predetermined quantity.
Moreover, there are restrictions on the size of annual investments.
Furthermore, annual production volumes are random variables, so the
need to take into account the probabilistic nature of the source
data can ruin the problem so that the use of CPLEX will become
inappropriate.

In future research, we plan to take into account the additional
restrictions and  specifics, as well as to develop a more efficient
approximate algorithm based on a genetic algorithm in which an
effective local search, for example, VNS \cite{Hansen01}, will be
used at the mutation stage.


\begin{thebibliography}{8}

\bibitem{Akopov04}  Akopov, A.S.: Metodi povisheniya effektivnosti upravleniya neftegazodobivaiushimi
ob'edinenuyami. Ekonomicheskaya nauka sovremennoy Rossii.\textbf{4},
88--99 (2004) (in Russian)

\bibitem{Brandt04}  Brandt, , M.W., Santa-Clara, P.: Dynamic Portfolio Selection by Augmenting the Asset Space.
NBER Working Paper. \textbf{10372}, JEL No. G0, G1 (2004)

\bibitem{Brandt05}  Brandt, M.W., Goyal, A., Santa-Clara, P., Stroud, J.R.:
A simulation approach to dynamic portfolio choice with an application
to learning about return predictability. Rev. Financ. Stud. \textbf{18}, 831-–873 (2005)

\bibitem{Brennan97} Brennan, M., Schwartz, E., Lagnado, R.: Strategic asset allocation.
J. Econ. Dyn. Control. \textbf{21}, 1377–-1403 (1997)

\bibitem{Bulai18} Bulai, V.C., Horobet, A.: A portfolio optimization model for a
large number of hydrocarbon exploration projects. In: Proc.
12 Int. Conf. on Business Excellence. \textbf{12}(1), 171-- 181 (2018). https://doi.org/10.2478/picbe-2018-0017

\bibitem{Cox86} Cox, J.C., Huang, C.-F.: Optimal consumption and portfolio policies
when asset prices follow a diffusion process.
J. Econ. Theory \textbf{49}, 33–-83 (1989)

\bibitem{Detemple03} Detemple, J.B., Garcia, R., Rindisbacher, M.: A Monte-Carlo method for optimal portfolios.
J. Finance. \textbf{58}, 401-–446 (2003)

\bibitem{Detemple14} Detemple, J.: Portfolio Selection: A Review. J. Optim. Theory Appl. \textbf{161}, 1--21 (2014)

\bibitem{Dominikov17} Dominikov, A., Khomenko, P., Chebotareva, G., and Khodorovsky, M.:
Risk and profitability optimization of investments
in the oil and gas industry. Int. J. of Energy Production and Management. \textbf{2}(3), 263--276 (2017)

\bibitem{Goncharenko08} Goncharenko, S.N., Safronova, Z.A.: Modeli i metody optimizacii
plana dobychi i pervichnoy pererabotki nifti.
Gorniy informacionno-analiticheskiy biulleten'. \textbf{10}, 221--229 (2008) (in Russian)

\bibitem{Karatzas87}    Karatzas, I., Lehoczky, J.P., Shreve, S.E.: Optimal portfolio and
consumption decisions for a ``mall investor'' on a
finite horizon. SIAM J. Control Optim. \textbf{25}, 1557-–1586 (1987)

\bibitem{Konovalov06}   Konovalov, E.N., Oficerov, V.P., Smirnov, S.V.: Povishenie
effectivnosti investiciy v neftedobyche na osnove modelirovaniya.
Problemi upravleniya i modelirovaniya v slognih sistemsh: Trudi V negdunarodnoy
konferencii. Samara. 381--385 (2006) (in Russian)

\bibitem{Hansen01} Hansen, P., Mladenovic, N.: Variable neighborhood search: Principles and applications,
European Journal of Operational Research. \textbf{130}, 449–-467 (2001)

\bibitem{Huang19} Huang, S.: An Improved Portfolio Optimization Model for Oil and Gas
Investment Selection Considering Investment Period.
Open Journal of Social Sciences. \textbf{7}, 121--129 (2019)

\bibitem{Malah16}   Malah, S.A., Servah, V.V.: O slognosti zadachi vybora
investicionnih proektov. Vestnik Omskogo universiteta. \textbf{3},
10–-15 (2016) (in Russian)

\bibitem{Markowitz52} Markowitz, H.M.: Portfolio Selection. J. of Finance. \textbf{7}(1), 71--91 (1952)

\bibitem{Markowitz59} Markowitz, H.M.: Portfolio Selection: Efficient Diversification
of Investment. –- Wiley, New York, 1959

\bibitem{Merton69} Merton, R.C.: Lifetime portfolio selection under uncertainty:
the continuous time case. Rev. Econ. Stat. \textbf{51}, 247-–257 (1969)

\bibitem{Merton71} Merton, R.C.: Optimum consumption and portfolio rules in a continuous-time model.
J. Econ. Theory \textbf{3}, 273-–413 (1971)

\bibitem{Ocone91}   Ocone, D., Karatzas, I.: A generalized Clark representation formula,
with application to optimal portfolios. Stoch. Stoch.
Rep. \textbf{34}, 187–-220 (1991)

\bibitem{Pliska86}  Pliska, S.: A stochastic calculus model of continuous trading:
optimal portfolios. Math. Oper. Res. \textbf{11}, 371-–382 (1986)

\bibitem{Power92} Power, M. Lognormality in the observed size distribution of oil
and gas pools as a consequence of sampling bias. Math
Geol 24, 929–-945 (1992). https://doi.org/10.1007/BF00894659

\bibitem{Qing14}     Qing, X., Zhen, W.,  Sijing, L.,  Dong, Z.: An improved portfolio
optimization model for oil and gas investment
selection. Pet. Sci. \textbf{11}, 181--188 (2014)

\bibitem{Skopina14} Skopina, L.V., Shubnikov, N.E.: Metodicheskiy podhod k ocenke
ionvesticionnih proektov v neftedobyche v usloviyah
neopredelennosti i riskov. Vestnik NGU. Seriya: Social'no-ekonomicheskie nauki.
\textbf{14}(2), 24--37 (2014) (in Russian)









\end{thebibliography}
\end{document}